\documentclass[11pt,a4paper]{article}
\usepackage{amsmath,amssymb,amsthm,graphicx,braket,multirow,authblk,amsthm,dcolumn,latexsym,subcaption,mathtools,cite,cancel}
\usepackage[normalem]{ulem}
\usepackage{subcaption}
\usepackage[a4paper]{geometry}
\usepackage{hyperref}
\usepackage{enumitem}
\usepackage{bm}

\usepackage{tikz}
\usepackage{pgfplots}
\pgfplotsset{compat=1.18}

\newtheorem{theorem}{Theorem}[section]
\newtheorem{proposition}[theorem]{Proposition}
\newtheorem{remark}[theorem]{Remark}

\newtheorem{corollary}[theorem]{Corollary}
\theoremstyle{definition}

\newtheorem{con}[theorem]{Conjecture}
\def\<{\langle}
\def\>{\rangle}

\newcommand{\Md}{M_d}
\newcommand{\Tr}{\operatorname{Tr}}

\newcommand{\I}{\mathbb{I}}
\date{}

\begin{document}

\title{{\bf Tomiyama-type maps with a diagonal perturbation}} 

\author{ Anindita Bera,$^1$ Bihalan Bhattacharya,$^2$ and Dariusz Chru{\'s}ci{\'n}ski$^2$}
\affil{$^1$Department of Mathematics, Birla Institute of Technology Mesra,\\ Jharkhand 835215, India\\
$^2$Institute of Physics, Faculty of Physics, Astronomy and Informatics, \\Nicolaus Copernicus University, Grudzi\c{a}dzka 5/7, 87--100 Toru{\'n}, Poland}
 

\maketitle

\begin{abstract}
\noindent

We investigate a two-parameter family of linear maps on matrix algebras, constructed as diagonal perturbations of classical Tomiyama maps. Employing the Choi matrix method alongside block-positivity techniques, we derive explicit necessary and sufficient conditions for positivity, complete positivity, and $k$-positivity across arbitrary dimensions. These conditions provide a transparent geometric characterization of the positivity regions within the parameter space.

\end{abstract}



\section{Introduction}

Positive linear maps on matrix algebras are central objects in operator theory \cite{Paulsen2002,Bhatia, Stormer2013} and quantum information \cite{QIT,Watrous, HHHH}. In particular, completely positive maps represent admissible physical transformations whereas positive maps which are not completely positive play a key role in characterizing quantum entangled states of composite systems \cite{HHHH}.

It is already known that the notion of complete  positivity is fully characterized by the Choi theorem \cite{Choi1975}. However, characterization of positive and $k$ positive (for $k>1$) maps is notoriously difficult \cite{Stormer1963,Choi72, Woronowicz, Tomiyama,Tomiyama-3,Kye,Kye-2,Kye-3, MM, Topical,Osaka,bera1,bera2,CMP}. Various constructions of $k$-positive maps were recently analyzed in \cite{CMP,Marciniak, Sun,Frederik}.
Let $M_d$ denote an algebra of $d\times d$ complex matrices and we consider linear maps $\Phi:M_d \to M_d$.  
In Ref. \cite{Tomiyama-II} Jun Tomiyama provided the following classic results:

\begin{theorem}[\cite{Tomiyama-II}] A linear map 
\begin{equation}
    \Phi_{\alpha} = (1-\alpha)\, \mathcal{I} + \alpha \,\tau_0 , ~~~~~-\infty<\alpha <\infty,
\end{equation}
where $\mathcal{I}$ denotes an identity map and $\tau_0$ stands for the normalized trace, i.e. $\tau_0(X) = \frac{\I}{d} \Tr\, X$, is $k$-positive iff

\begin{equation}
    0 \leq \alpha \leq \alpha_k :=  \frac{kd}{kd-1} ,
\end{equation}
for $k=1,\ldots,d$. 
\end{theorem}

\begin{theorem}[\cite{Tomiyama-II}] A linear map

\begin{equation}
    \Lambda_{\mu} = (1-\mu)\, T + \mu \,\tau_0 ,
\end{equation}
where $T$ denotes transposition, is positive iff

\begin{equation}
    0 \leq \mu \leq  \frac{d}{d-1} ,
\end{equation}
and $k$-positive for $k=2,\ldots,d$ iff

\begin{equation}
     \frac{d}{d+1} \leq \mu \leq \frac{d}{d-1} .
\end{equation}
\end{theorem}
In what follows we call  $\mathcal{T}_k := \Phi_{\alpha_k}$ Tomiyama maps.
In this paper, we generalize Tomiyama results for the following 2-parameter class of maps

\begin{equation}
    \Phi_{\alpha,\beta} = (1-\alpha-\beta)\, \mathcal{I} + \alpha \,\tau_0 + \beta\, \Delta ,~~~\alpha,\beta \in \mathbb{R}
    \label{defn_phi}
\end{equation}
and

\begin{equation}
    \Lambda_{\mu,\nu} = (1-\mu-\nu)\, {T} + \mu \,\tau_0 + \nu\, \Delta, ~~~\mu,\nu \in \mathbb{R}
\end{equation}
where $\Delta(X)$ denotes a projection to the diagonal of $X \in \Md$. Note, that both families $\Phi_{\alpha,\beta}$ and $\Lambda_{\mu,\nu}$ are self-adjoint w.r.t. Hilbert-Schmidt inner product in $M_d$. Moreover, they are both unital and trace preserving. Hence, both families provide a diagonal perturbation of the maps considered in \cite{Tomiyama-II}.

The paper is organized as follows. In Section \ref{SI}, we  derive complete characterizations of positivity and complete positivity of  $\Phi_{\alpha,\beta}$. Moreover,  we analyze $k$-positivity and the associated geometric conjecture. 
In Section \ref{SII}, we analyze maps $\Lambda_{\mu,\nu}$ and derive the necessary and sufficient conditions for $k$-positivity. The paper concludes with a summary of the results and an outlook on possible extensions in Section \ref{SIV}. 

\section{A class of linear maps $\Phi_{\alpha,\beta}$}   \label{SI}

\subsection{Positivity and complete positivity}

In this section, we derive the conditions for positivity and complete positivity of the map $\Phi_{\alpha,\beta}$ . 

\begin{proposition} \label{PCP} A linear map $\Phi_{\alpha,\beta}$ defined in Eq.~\eqref{defn_phi} is 

\begin{enumerate}
    \item positive  iff

\begin{equation}   \label{1P}
    0 \leq \alpha \leq \frac{d}{d-1} \ , \quad - \frac{2\alpha}{d} \leq \beta \leq \frac{d}{d-1} - \alpha , 
\end{equation}

\item completely positive iff

\begin{equation}  \label{CP}
    0 \leq \alpha \leq \frac{d}{d-1} \ , \quad - \frac{\alpha}{d} \leq \beta \leq \frac{d}{d-1} - \frac{d+1}{d}\, \alpha .
\end{equation}
    
\end{enumerate} 
    
\end{proposition}

\begin{proof}
We first show the conditions for complete positivity. The corresponding Choi matrix \cite{Choi1975} of $\Phi_{\alpha,\beta}$ reads

\begin{equation}
    C_{\alpha,\beta} = \sum_{i,j} e_{ij} \otimes \Phi_{\alpha,\beta}(e_{ij}) ,
\end{equation}
where $e_{ij}$ are matrix units in $\Md$. One finds

\begin{equation}   \label{Choi}
    C_{\alpha,\beta} = \gamma \sum_{i,j} e_{ij} \otimes e_{ij} + \frac{\alpha}{d} \I_d\otimes \I_d + \beta \sum_i e_{ii} \otimes e_{ii} ,
\end{equation}
where $\gamma = 1-\alpha - \beta$. Now, $\Phi_{\alpha,\beta}$ is positive iff $C_{\alpha,\beta}$ is block-positive, and it is completely positive iff the Choi matrix is positive definite. Any block-positive matrix has non-negative diagonal elements and hence

$$ \gamma + \frac{\alpha}{d} + \beta = 1 - \frac{d-1}{d}\alpha \geq 0 \ , \quad \frac{\alpha}{d} \geq 0 ,   $$
which is equivalent to

\begin{equation}
     0 \leq \alpha \leq \frac{d}{d-1} .
\end{equation}
Taking into account the following eigenvalues of the Choi matrix (\ref{Choi})

\[ \frac{\alpha}{d}+\beta,\qquad d-\frac{d^2-1}{d}\alpha-(d-1)\beta,
\]
one immediately proves (\ref{CP}).  

Now, to derive positivity condition (\ref{1P}), let us consider the action of the map $\Phi_{\alpha,\beta}$ on $X \in \Md$ such that $X_{ij}=1$ for all $i,j=1,\ldots,d$. One finds

\begin{equation}
    \Phi_{\alpha,\beta}(X) = \gamma X + (\alpha + \beta)\, \I_d, 
\end{equation}
with eigenvalues $\{\alpha+\beta, d - (d-1)(\alpha+\beta)\}$ and hence positivity of the map implies 

$$ 0 \leq \alpha + \beta \leq \frac{d}{d-1} ,  $$
which in turn implies the upper bound on $\beta $ in (\ref{1P}). On the other hand taking $X = vv^\dagger$, with $v = e_1 - e_d$ one obtains

 \[\Phi_{\alpha,\beta}(X)=\begin{pmatrix}
    1-\tfrac{(d-2)\alpha}{d}&0&0\cdots0&0&-1+\alpha+\beta\\
    0&\tfrac{2\alpha}{d}&0\cdots0&0&0\\
    \vdots & \vdots& \ddots&\vdots&\vdots\\
    0&0&0\cdots0&\tfrac{2\alpha}{d}&0\\
    -1+\alpha+\beta&0&0\cdots0&0&1-\tfrac{(d-2)\alpha}{d}\\
\end{pmatrix} ,\]
and hence positivity of $\Phi_{\alpha,\beta}(X)$ implies $2\alpha+d\beta \geq 0$ which in turn implies the lower bound $-2\alpha/d \leq \beta$ in (\ref{1P}). This way we proved that (\ref{1P}) provides a necessary condition. To show that it is also sufficient we prove that if these bounds (lower or upper) are saturated then the map is positive. Indeed, suppose that $\beta = -2\alpha/d$. Then the Choi matrix (\ref{Choi}) may be represented as follows

\[ C = A + (\mathcal{I} \otimes T)B_{+} ,\]
where
\[ A = \left(1- \frac{d-1}{d}\alpha\right) \sum_{i,j} e_{ij} \otimes e_{ij} \ , \quad B_+ = \frac{\alpha}{d} \,\sum_{i<j} f_{ij}\, f_{ij}^\dagger\ , \]
and $f_{ij} = (e_i + e_j) \otimes (e_i + e_j)$. Now, since $A$ and $B_{+}$ are evidently positive definite $C$ is block-positive.  Similarly, if $\beta = \frac{d}{d-1} - \alpha$, then 

\[ C = A + (\mathcal{I} \otimes T)B_- ,\]
where 
\[  B_- = \frac{\alpha}{d} \,\sum_{i<j} g_{ij} \,g_{ij}^\dagger\ , \]
and $g_{ij} = (e_i - e_j) \otimes (e_i - e_j)$. Again, block positivity of $C$ immediately follows which completes the proof. 
\end{proof}

\begin{corollary} All positive maps $\Phi_{\alpha,\beta}$ are decomposable. 
\end{corollary}

 Let us define completely positive maps corresponding to the particular points $(\alpha,\beta)$:

\[   \Psi_0   \leftrightarrow (0,0) \ , \Psi_1 \leftrightarrow \left(0,\frac{d}{d-1}\right) \ , \ \Psi_2 \leftrightarrow \left(\frac{d}{d-1},-\frac{1}{d-1}\right) .
\]
One finds

\begin{eqnarray*}
    \Psi_0(X) &=& X \  , \\
    \Psi_1(X) &=& \frac{1}{d-1} \Big( d \Delta(X) - X) , \\
     \Psi_2(X) &=& \frac{1}{d-1} \Big( \I_d\Tr X  - \Delta(X) \Big) ,
\end{eqnarray*}
Similarly define positive maps

\[  \mathcal{T}_1 \leftrightarrow \left(\frac{d}{d-1},0\right) \ , \ \mathcal{P} \leftrightarrow \left(\frac{d}{d-1},-\frac{2}{d-1}\right) \ , 
\]
 that is,
\begin{eqnarray*}
    \mathcal{T}_1(X) &=& \frac{1}{d-1} ( \I_d\Tr X - X) , \\
     \mathcal{P}(X) &=& \frac{1}{d-1} \Big( \I_d\Tr X  + X - 2\Delta(X) \Big) .
\end{eqnarray*}
Note, that the map $\mathcal{T}_1$ corresponding to a Tomiyama points $(\alpha_1,0)$ is nothing but the celebrated reduction map \cite{HHHH} which plays a key role in entanglement theory.

\begin{corollary} \label{cor1} A set of completely positive maps $\Phi_{\alpha,\beta}$ forms a triangle (cf. Fig. \ref{FIG1})

$$  \mathcal{P}_{\rm CP} = {\rm conv} \{\Psi_0,\Psi_1,\Psi_2\} , $$
whereas a set of positive maps forms a quadrilateral containing CP triangle (cf. Fig. \ref{FIG2})

$$ \mathcal{P}_1 =  {\rm conv} \{\Psi_0,\Psi_1,\mathcal{T}_1,\mathcal{P}\} . $$
Here `conv' stands for the convex hull. Note that $\Psi_2 \in {\rm conv} \{\mathcal{T}_1,\mathcal{P}\}$. Indeed, $\Psi_2 = \frac 12 (\mathcal{T}_1 + \mathcal{P})$.

\end{corollary}

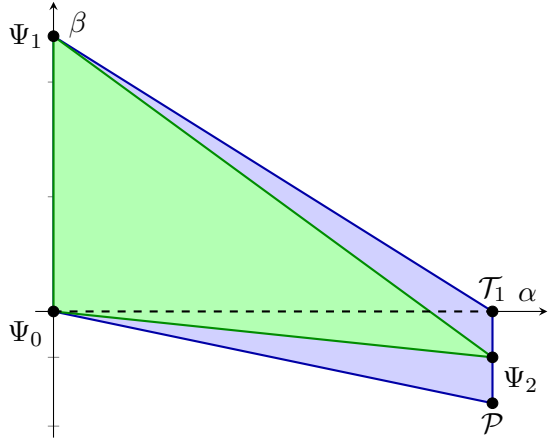
\begin{figure}[htbp]
\centering
\begin{tikzpicture}

\begin{axis}[
    axis lines=middle,
    xmin=-0.05, xmax=1.35,
    ymin=-0.55, ymax=1.35,
    xlabel={$\,\alpha$},
    ylabel={$\,\beta$},
    xtick={0,0.5,1},
    ytick={-0.5,-0.2,0,0.5,1},
    yticklabels={}, 
    width=0.52\textwidth,
    height=0.46\textwidth,
    legend style={at={(0.02,0.98)},anchor=north west,draw=none,fill=none},
    clip=false,
]

\def\a{1.2}  
\def\b{0.2}  

\addplot[
  fill=blue!18,
  draw=blue!65!black,
  thick
] coordinates {
  (0,0)
  (0,\a)
  (\a,0)
  (\a,-2*\b)
  (0,0)
};

\addplot[
  fill=green!30,
  draw=green!55!black,
  thick
] coordinates {
  (0,0)
  (0,\a)
  (\a,-\b)
  (0,0)
};

\addplot[only marks,mark=*,mark size=2pt] coordinates {
  (0,0)
  (0,\a)
  (\a,-\b)
  (\a,0)
  (\a,-2*\b)
};

\node[below left]  at (0,0) {$\Psi_0$};
\node[left]        at (0,\a) {$\Psi_1$};
\node[below right] at (\a,-\b) {$\Psi_2$};
\node[above]       at (\a,0) {$\mathcal{T}_1$};
\node[below]       at (\a,-2*\b) {$\mathcal P$};

\addplot[
  black,
  thick,
  dashed
]
coordinates {
  (0,0)
  (\a,0)
};

\end{axis}
\end{tikzpicture}
\caption{Illustration of Corollary~2.3 in the $(\alpha,\beta)$-plane.
The completely positive region $\mathcal P_{\mathrm{CP}}$ is the green triangle
$\mathrm{conv}\{\Psi_0,\Psi_1,\Psi_2\}$, and the positive region $\mathcal P_1$
is the quadrilateral $\mathrm{conv}\{\Psi_0,\Psi_1,\mathcal{T}_1,\mathcal P\}$.
Moreover, $\Psi_2=\tfrac12(\mathcal{T}_1+\mathcal P)$ lies on the vertical segment between $\mathcal{T}_1$ and $\mathcal P$.   \label{FIG1}}
\end{figure}

\subsection{Analysis of $k$-positivity}

Let $\mathcal{T}_k$ be a $k$-positive Tomiyama map corresponding to $(\alpha_k,0)$.

\begin{con} \label{con} A set of $k$-positive maps $\Phi_{\alpha,\beta}$ forms a quadrilateral 

$$ \mathcal{P}_k =  {\rm conv} \{\Psi_0,\Psi_1,\Psi_2,\mathcal{T}_k\} . $$
    
\end{con}
Since CP maps $\{\Psi_0,\Psi_1,\Psi_2\}$ are $k$-positive the above set contains $k$-positive maps only. Our goal is to show that if $\Phi_{\alpha,\beta}$ is $k$-positive it must belong to $\mathcal{P}_k$. For $k=1$ it reduces to Corollary \ref{cor1}. 

Let

\[ \mathcal{P}^+_k =  {\rm conv} \{\Psi_0,\Psi_1,\mathcal{T}_k\} ,  \quad \mathcal{P}^-_k =  {\rm conv} \{\Psi_0,\Psi_2,\mathcal{T}_k\} ,   \]
that is, $\mathcal{P}^+_k$ contains $k$-positive maps with $\beta \geq 0$ and  $\mathcal{P}^-_k$ contains $k$-positive maps with $\beta \leq 0$. Clearly one has $ \mathcal{P}_k =  \mathcal{P}^+_k \cup  \mathcal{P}^-_k$.

\begin{proposition} Let $\Phi_{\alpha,\beta}$ be a $k$-positive map with $\beta \leq 0$. Then $\Phi_{\alpha,\beta} \in \mathcal{P}^-_k$.
\end{proposition}

\begin{proof}
    Consider a vector  $v \in \mathbb{C}^k \otimes \mathbb{C}^d$ defined by

\[  v = e_1 \otimes e_1 + \ldots + e_k \otimes e_k . \]
If $ \Phi_{\alpha,\beta}$ is $k$-positive, then 

\[  (\mathcal{I}_k \otimes \Phi_{\alpha,\beta})(vv^\dagger) \geq 0 , \]
where $\mathcal{I}_k$ is an identity map in $\mathbb{M}_k$. One finds

\[  (\mathcal{I}_k \otimes \Phi_{\alpha,\beta})(vv^\dagger) = \sum_{i,j=1}^k e_{ij} \otimes X_{ij} , \]
where the blocks $X_{ij} \in \Md$ read

\[  X_{ii} = \left( 1 - \frac{d-1}{d} \alpha \right) e_{ii} + \frac{\alpha}{d}\, (\I_d - e_{ii})\ , \quad X_{ij} = (1 - \alpha - \beta) e_{ij} \ ; (i \neq j) . \]
One easily finds the corresponding eigenvalues of $(\mathcal{I}_k \otimes \Phi_{\alpha,\beta})(vv^\dagger)$:

\[ \left\{ \frac{\alpha}{d}\ , \frac{\alpha}{d}+\beta \ , kd-(kd-1)\alpha-d(k-1)\beta \right\}\ .\]
Now observe that the line 

\[ \frac{\alpha}{d}+\beta = 0 ,\]
connects $\Psi_0$ and $\Psi_2$, whereas the line

\[ kd-(kd-1)\alpha-d(k-1)\beta = 0 , \]
connects $\Psi_2$ and $\mathcal{T}_k$. Indeed, if $\beta=0$ then $\alpha = \alpha_k = \frac{dk}{dk-1}$ which corresponds to $\mathcal{T}_k$, and if $\beta=- \frac{1}{d-1}$, then $\alpha= \frac{d}{d-1}$ which corresponds to $\Psi_2$. This shows that indeed any $k$-positive $\Phi_{\alpha,\beta}$ with $\beta\leq 0$ belongs to  $\mathcal{P}^-_k$. \end{proof}

\begin{figure}[htbp]
\centering
\begin{tikzpicture}
\begin{axis}[
    axis lines=middle,
    xmin=-0.05, xmax=1.45,
    ymin=-0.55, ymax=1.35,
    xlabel={$\,\alpha$},
    ylabel={$\,\beta$},
    xtick={0,1},
    ytick=\empty, 
    width=0.52\textwidth,
    height=0.46\textwidth,
    legend style={at={(0.02,0.98)},anchor=north west,draw=none,fill=none},
    clip=false,
]

\pgfmathsetmacro{\d}{6}
\pgfmathsetmacro{\k}{2}

\pgfmathsetmacro{\a}{\d/(\d-1)}          
\pgfmathsetmacro{\b}{1/(\d-1)}           
\pgfmathsetmacro{\ak}{1 + 1/(\k*\d - 1)} 

\addplot[
  fill=green!30,
  draw=green!55!black,
  thick
] coordinates {
  (0,0)        
  (0,\a)       
  (\a,-\b)     
  (0,0)
};

\addplot[
  fill=blue!22,
  draw=blue!65!black,
  thick
] coordinates {
  (0,\a)       
  (\a,-\b)     
  (\ak,0)      
  (0,\a)
};

\addplot[
  black,
  thick,
  dashed
] coordinates {
  (0,0)
  (\ak,0)
};

\addplot[only marks,mark=*,mark size=2pt] coordinates {
  (0,0)        
  (0,\a)       
  (\a,-\b)     
  (\ak,0)      
};

\node[below left]  at (0,0) {$\Psi_0$};
\node[left]        at (0,\a) {$\Psi_1$};
\node[below right] at (\a,-\b) {$\Psi_2$};
\node[above right] at (\ak,0) {$\mathcal{T}_k$};

\end{axis}
\end{tikzpicture}
\caption{Illustration of Conjecture~3.1.
The completely positive region $\mathcal P_{\mathrm{CP}}=\mathrm{conv}\{\Psi_0,\Psi_1,\Psi_2\}$ is shown in green.
The additional triangle $\mathrm{conv}\{\Psi_1,\Psi_2,\mathcal{T}_k\}$ (blue) represents the conjectured extension to the set of $k$-positive maps.    \label{FIG2}
}
\end{figure}
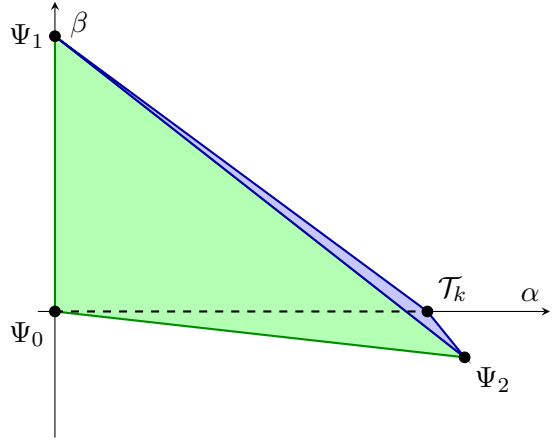
Note, that $\mathcal{T}_d$ belongs to the segment ${\rm conv}\{\Psi_1,\Psi_2\}$. Indeed, one has

\begin{equation}
    \mathcal{T}_d(X) = \frac{1}{d^2-1} \Big( d\,\I_d \Tr X - X \Big) ,
\end{equation}
and hence 
\begin{equation}
    \mathcal{T}_d = \frac{1}{d+1} \Big( \Psi_1 + d\, \Psi_2 \Big) . 
\end{equation}

\begin{proposition} If $d=\ell k$, i.e. $k$ divides $d$, then Conjecture \ref{con} holds. 
    
\end{proposition}

\begin{proof}  
  {It is sufficient to prove that for $\beta \geq 0$ and $d=\ell k$, $k$-positivity region for the map $\Phi_{\alpha,\beta}$ is $\mathcal{P}^+_k$.}  
  Consider a vector  $w \in \mathbb{C}^k \otimes \mathbb{C}^{\ell k}$ defined by

\[  w = e_1 \otimes w_1 + \ldots + e_k \otimes w_k , \]
with
\[ w_1=f_1+ \ldots + f_\ell \ , \ w_2= f_{\ell+1} + \ldots + f_{2\ell} \ , \ldots \ , 
  w_k = f_{(k-1)\ell + 1} +  \ldots + f_{k\ell}  , \]
where $\{f_1,\ldots,f_{\ell k}\}$ denotes a canonical basis in $\mathbb{C}^{\ell k}$. 
If $ \Phi_{\alpha,\beta}$ is $k$-positive, then $(\mathcal{I}_k \otimes \Phi_{\alpha,\beta})(ww^\dagger) \geq 0$. 
One finds

\[  (\mathcal{I}_k \otimes \Phi_{\alpha,\beta})(ww^\dagger) = \sum_{i,j=1}^k e_{ij} \otimes X_{ij} , \]
where the blocks $X_{ij} \in \Md$ read

\[ X_{ii} =  (1 - \alpha - \beta) w_i w_i^\dagger +  \frac{\alpha}{k} \, \I_{d} + \beta \sum_{j = (i-1)k+1}^{(i-1)k + \ell} f_{jj}   \ ,   \]
where $f_{ij} = f_i f_j^\dagger$ and $ X_{ij} = (1 - \alpha - \beta) \, w_i w^\dagger_j$ for $i \neq j$.
Consider now a principal $k\ell \times k\ell$ submatrix defined by

\[  W= \sum_{i,j=1}^k e_{ij} \otimes W_{ij} \ , \]
with diagonal blocks

\[ W_{ii} := (1 - \alpha - \beta)\, w_i w_i^\dagger + \left( \frac{\alpha}{ k} + \beta\right) \!\! \sum_{j = (i-1)k+1}^{(i-1)k + \ell} f_{jj} \ , \]
and off-diagonal blocks $W_{ij} := X_{ij}$.  Being a principal submatrix of a positive $k^2\ell \times k^2\ell$ matrix $(\mathcal{I}_k \otimes \Phi_{\alpha,\beta})(ww^\dagger)$ it must be positive as well.  Note that $W$ has the constant row sum

\[\lambda = 1-\frac{(k-1)\alpha}{k}+(\ell-1)(1-\alpha-\beta)+\ell(k-1)(1-\alpha-\beta) ,\]
and hence $\lambda$ defines its eigenvalue. 
{If $ \Phi_{\alpha,\beta}$ is $k$-positive then $\lambda \geq 0$, which after simplification gives 
%
\begin{equation} \label{dk}
  (kd-1)\alpha+k(d-1)\beta\leq kd  .   
\end{equation}
Note that the line joining  maps $\Psi_1$ and $\mathcal{T}_k$ is given by $(kd-1)\alpha+k(d-1)\beta = kd$. Both maps $\Psi_1$ and $\mathcal{T}_k$ are $k$-positive and hence (\ref{dk}) implies that for $\beta \geq 0$ and $d=\ell k$, the $k$-positivity region of $\Phi_{\alpha,\beta}$ is $\mathcal{P}^+_k$.}
\end{proof}

\begin{corollary} If $d$ is even, then any 2-positive map $\Phi_{\alpha,\beta}$ belongs to $\mathcal{P}_2$.
    
\end{corollary}

\begin{proposition} If $k=d-1$, then Conjecture \ref{con} holds. 
    
\end{proposition}
\begin{proof}

One needs to show that the line connecting $\Psi_1$ and $\mathcal{T}_k$ defines the boundary of a set of $k$-positive maps (for $\beta \geq 0$).  The line is defined by

\begin{equation}   \label{beta}
    \beta = \frac{k+1}{k} - \frac{k^2+k+1}{k^2}\, \alpha . 
\end{equation}
Define the following vector $x \in \mathbb{C}^k \otimes \mathbb{C}^{k+1}$

\begin{equation}
    x = \sum_{i=1}^k e_i \otimes x_i \ , \ \ x_i = f_i + f_{i+1} .
\end{equation}
One finds

\begin{equation}
  X =  ({\rm id}_k \otimes \Phi_{\alpha,\beta})(|x\>\<x|) = \sum_{i,j=1}^k e_{ij} \otimes X_{ij} .
\end{equation}
Inserting $\beta$ from (\ref{beta}) one gets

$$ X = A + \alpha\, B , $$
where $A = \sum_{i,j=1}^k e_{ij} \otimes A_{ij}$ and $B=\sum_{i,j=1}^k e_{ij} \otimes B_{ij}$ are block matrices with
$d \times d$ blocks $A_{ij}$ and $B_{ij}$  defined as follows:

\begin{eqnarray}
    A_{ii} &=&  - \frac 1k |x_i\>\<x_i| + \frac{k+1}{k} ( f_{ii} + f_{i+1,i+1} ) \ , \nonumber \\
    A_{i,i+1} &=&  - \frac 1k |x_i\>\<x_{i+1}| +  \frac{k+1}{k} \, f_{i+1,i+1} \ , \\
    A_{ij} &=&  - \frac 1k |x_i\>\<x_j|  \ , \ \ \   (j > i+1) , \nonumber
\end{eqnarray}
and 

\begin{eqnarray}
    B_{ii} &=&  \frac{k-1}{k^2} (f_{i,i+1}+f_{i+1,i}) + \frac{k-1}{k+1} \I_d - (f_{ii} + f_{i+1,i+1})   \ , \nonumber \\
    B_{i,i+1} &=&  \frac{k-1}{k^2} |x_i\>\<x_{i+1}| +  \frac{1}{k+1}\I_d -  f_{i+1,i+1} \ , \\
    B_{ij} &=&  \frac{k-1}{k^2} |x_i\>\<x_j|  \ , \ \ \   (j > i+1) . \nonumber
\end{eqnarray}
To finish the proof, it suffices to show that $X = A + \alpha B$ has a zero eigenvalue. Interestingly, it turns out that there exists a zero mode $\psi\in \mathbb{C}^k \otimes \mathbb{C}^{k+1}$ independent of $\alpha$, that is,

$$   A \psi= 0 \ , \quad B\psi=0 . $$
The structure of $A$ and $B$ implies a {\em mirror} symmetry of $\psi$, i.e. 

\[ (\psi_i)_\ell =  (\psi_{k-i})_{d-\ell} \ , \ \ i=1,2,\ldots,k \ ; \ \ell=1,\ldots,d . \]
Let $\psi = \sum_{i=1}^k e_i \otimes \psi_i$. Define 

\begin{equation}
    m_j = (-1)^{j+1} (d-j) .
\end{equation}
Then for $k = 2n$ (equivalently $d=2n+1$) one has 
\begin{equation}
    \psi_j = \Big(\underbrace{m_j,-m_j,m_j,\ldots,(-1)^{j+1} m_j}_{j},\underbrace{j,-j,j,\ldots,(-1)^{j+1} j}_{d-j} \Big)^{\rm T} 
\end{equation}
for $j =1, \ldots,n$. For $k = 2n-1$ (equivalently $d=2n$) one has

\begin{equation}
    \psi_j = \Big(\underbrace{m_j,-m_j,m_j,\ldots,(-1)^j m_j}_{j},\underbrace{j,-j,j,\ldots,(-1)^j j}_{d-j} \Big)^{\rm T} 
\end{equation}
for $j =1, \ldots,n$ and

\begin{equation}
    \psi_{n+1} = \Big(\underbrace{n,-n,n,\ldots,-n,n}_{n},\underbrace{n,-n,\ldots, -n,n }_{n} \Big)^{\rm T} 
\end{equation}
for odd $n$, and 

\begin{equation}
    \psi_{n+1} = \Big(\underbrace{-n,n,-n,\ldots,-n,n}_{n},\underbrace{n,-n,\ldots, -n,n,-n }_{n} \Big)^{\rm T} 
\end{equation}
for even $n$. One checks that indeed $\psi$ defines a zero mode for $A$ and $B$ and hence it defines a zero mode for $X = A + \alpha B$. \end{proof}
In Appendix we illustrate the construction of matrices $A$ and $B$ together with a zero mode $\psi$ for $k=3$ and $k=4$.  

\begin{remark}
{\em   In Ref. \cite{Cho} authors considered the map $\Phi_{[a,b,c]} : M_3 \to M_3$

\begin{equation}
    \Phi_{[a,b,c]}(X) = \Psi_{[a,b,c]}(X) - X , 
\end{equation}
where $\Psi_{[a,b,c]}(X)$ denotes the following diagonal matrix

\begin{equation*}
     \begin{pmatrix} (a+1) x_{11} + b x_{22} + c x_{33} & 0 & 0 \\ 0 &  (a+1) x_{22} + b x_{33} + c x_{11} & 0 \\ 0 & 0 &  (a+1) x_{33} + b x_{11} + c x_{22}         
    \end{pmatrix} ,
\end{equation*}
and $a,b,c \geq 0$ (note, that we changed parameterization from $a \to a+1$). This map is self-adjoint if and only if $b=c$ and in this case it belongs to the class $\Phi_{\alpha,\beta}$ (up to  a multiplication factor). Indeed, one finds

\[   \Phi_{[a,b,b]} = (a+2b)\, \Phi_{\alpha,\beta} \ , \]
with 
\[ a = \frac{1- \frac 23 \alpha}{\alpha+\beta - 1}\ , \quad b=c = \frac{\alpha}{3(\alpha+\beta - 1)} . \]
Theorem 4.2 in \cite{Cho} states that the above map is 2-positive if and only if $a \geq 2$ or the following condition is satisfied ((4.3) in \cite{Cho}):

$$  b^2 = (2-a)(b+c) > 0 .   $$
Now, since in our class $b=c$ one gets $b = 2(2-a)$. The condition $a = 2$ is equivalent to 

$$   8 \alpha + 6 \beta = 9 , $$
and defines the line passing through two CP maps $\Psi_1$ and $\Psi_2$ corresponding to $(0,\frac 32)$ and $(\frac 32,-\frac12)$ (cf. Fig. \ref{FIG-3}). On the other hand, condition $b = 2(2-a)$ is equivalent to 

$$   5 \alpha + 4 \beta = 6 , $$
and defines a line passing through a CP map $\Psi_1$ and a 2-positive map $\mathcal{T}_2$ corresponding to $(\frac 65,0)$. 
Hence they perfectly coincide with the conditions derived in the present paper. Note however that for $\beta < 0$ the line $5 \alpha + 4 \beta = 6$ goes outside the 2-positivity region. Hence, Theorem 4.2 from \cite{Cho} correctly characterizes 2-positivity for $b=c$ case only for $\beta \geq 0$, that is, using 

\[
\alpha = \frac{3b}{a+2b},
\qquad
\beta = \frac{a-b+1}{a+2b}.
\]
$\beta > 0$ is equivalent to $a>b-1$. Note, that on the line passing through $\Psi_1$ and $\mathcal{T}_2$ one has

$$ a=1 \ , \quad b = \frac{\alpha}{3-2\alpha} \ , \ \ \ \alpha \in \Big[\frac 65,\frac 32\Big) . $$
Summarizing: for $\beta \geq 0$ the 2-positive self-adjoint trace-preserving maps (i.e. $b=c$) from \cite{Cho} belong to the boundary of 2-positivity region $\mathcal{P}_2$. 

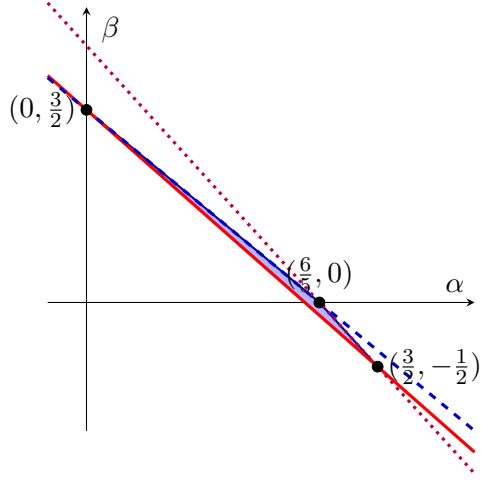
\begin{figure}
    \centering
 
\begin{tikzpicture}
\begin{axis}[
    axis lines=middle,
    xmin=-0.2, xmax=2,
    ymin=-1, ymax=2.3,
    xlabel={$\,\alpha$},
    ylabel={$\,\beta$},
    xtick=\empty,
    ytick=\empty,
    width=0.45\textwidth,
    height=0.45\textwidth,
    clip=false,
]

\addplot[
  fill=blue!30,
  draw=blue!65!black,
  thick
] coordinates {
  (0,3/2)
  (6/5,0)
  (3/2,-1/2)
  (0,3/2)
};

\addplot[
  red,
  very thick,
  domain=-0.2:2
] {3/2 - (4/3)*x};

\addplot[
  blue!80!black,
  very thick,
  dashed,
  domain=-0.2:2
] {3/2 - (5/4)*x};

\addplot[
  purple,
  very thick,
  dotted,
  domain=-0.2:2
] {2 - (5/3)*x};

\addplot[only marks,mark=*,mark size=2pt] coordinates {
  (0,3/2)
  (6/5,0)
  (3/2,-1/2)
};

\node[left]  at (axis cs:0,3/2) {$(0,\tfrac{3}{2})$};
\node[above] at (axis cs:6/5,0) {$(\tfrac{6}{5},0)$};
\node[right] at (axis cs:3/2,-1/2) {$(\tfrac{3}{2},-\tfrac{1}{2})$};

\end{axis}
\end{tikzpicture}
   \caption{A solid red line defines the border of CP region. The dashed blue line from \cite{Cho} belongs to the boundary of $\mathcal{P}_2$ for $\beta \geq 0$ but goes outside $\mathcal{P}_2$ for $\beta < 0$. The dotted red line defines the boundary of $\mathcal{P}_2$ for $\beta \leq 0$.  \label{FIG-3}}
\end{figure}}
\end{remark}

\section{A class  of linear maps $\Lambda_{\mu,\nu}$}  \label{SII}

In this Section, we analyze generalization of the class $\Lambda_{\mu,\nu}$.  

\begin{proposition}
The linear map $\Lambda_{\mu,\nu}$ is 

\begin{enumerate}
    \item positive iff

\begin{equation}
    0 \leq \mu \leq \frac{d}{d-1} \ , \ \ - \frac{2}{d} \mu \leq \nu \leq \frac{d}{d-1} - \mu \ ,
\end{equation}
\item $k$-positive for $k=2,\ldots,d$ iff

\begin{equation}  \label{2p}
      0 \leq \mu \leq \frac{d}{d-1} \ , \ \ 1 - \frac{d+1}{d}\mu \leq \nu \leq 1 - \frac{d-1}{d} \mu \ . 
\end{equation}

\end{enumerate}
    
\end{proposition}
\begin{proof} The proof of positivity (corresponding to $k=1$) is straightforward. Moreover, using similar argumantes as in the case of $\Phi_{\alpha,\beta}$ it is easy to shown that the above conditions are sufficient for 2-positivity.   To show that they are also necessary  let us consider a vector $w \in \mathbb{C}^2 \otimes \mathbb{C}^d $ given by 
\begin{equation*}
    w=r_1 \otimes e_1+r_2 \otimes e_d ,
\end{equation*}
where $\lbrace r_1, r_2 \rbrace$ and  $\mathbb{C}^d$ as $ \lbrace e_1, e_2,\ldots,e_d \rbrace$ are orthonormal bases of $\mathbb{C}^2$ and $\mathbb{C}^d$. The eigenvalues of $\mathcal{I}_2\otimes \Lambda_{\mu,\nu}(|w\rangle \langle w|)$ read

$$1-\frac{(d-1)}{d}\mu\  , \ 1-\frac{(d-1)}{d}\mu-\nu \ , \ \frac{(d+1)}{d}\mu + \nu-1, $$ 
 hence, positivity of $\mathcal{I}_2\otimes \Lambda_{\mu,\nu}(Y)$ implies (\ref{2p}). One easily checks that 2-positivity already implies complete positivity.

\end{proof}

\begin{corollary} All positive maps $\Lambda_{\mu,\nu}$ are decomposable. 
\end{corollary}

 Let us define completely positive maps corresponding to the particular points $(\mu,\nu)$:

\[   \widetilde{\Psi}_0   \leftrightarrow \left(\frac{d}{d-1},0\right) \ ,  \widetilde{\Psi}_1 \leftrightarrow \left(\frac{d}{d-1},\frac{-2}{d-1}\right) \ , \  \widetilde{\Psi}_2 \leftrightarrow \left(0,1\right) .
\]
One finds

\begin{eqnarray*}
    \widetilde{\Psi}_0(X) &=& \frac{1}{d-1} \Big( \I_d\Tr X  - X^{\rm T} \Big)\  , \\
    \widetilde{\Psi}_1(X) &=& \frac{1}{d-1} \Big( \I_d\Tr X + X^{\rm T} - 2\Delta(X)\Big) , \\
     \widetilde{\Psi}_2(X) &=& \Delta(X)  .
\end{eqnarray*}
Note, that $\widetilde{\Psi}_0$ is the celebrated Holevo-Werner channel \cite{Holevo,WH}.
Similarly define positive maps

\[  \widetilde{\mathcal{T}}_1 \leftrightarrow \left(0,0\right) \ , \ \widetilde{\mathcal{P}} \leftrightarrow \left(0,\frac{d}{d-1},\right) \ , 
\]
 that is,
\begin{eqnarray*}
    \widetilde{\mathcal{T}}_1(X) &=& X^{\rm T} , \\
     \widetilde{\mathcal{P}}(X) &=& \frac{1}{d-1} \Big( d\Delta(X) - X^{\rm T}\Big) .
\end{eqnarray*}
Note, that the map $\widetilde{\mathcal{T}}_1$ is the Tomiyama map $\Lambda_0$.  

\begin{corollary} \label{cor3} A set of completely positive maps $\Lambda_{\mu,\nu}$ forms a triangle (cf. Fig. \ref{FIG4})

$$  \widetilde{\mathcal{P}}_{\rm CP} = {\rm conv} \{\widetilde{\Psi}_0,\widetilde{\Psi}_1,\widetilde{\Psi}_2\} , $$
whereas a set of positive maps forms a quadrilateral containing CP triangle

$$ \widetilde{\mathcal{P}}_1 =  {\rm conv} \{\widetilde{\Psi}_0,\widetilde{\Psi}_1,\widetilde{\mathcal{T}}_1,\widetilde{\mathcal{P}}\} . $$
Note that $\widetilde{\Psi}_2 \in {\rm conv} \{\widetilde{\mathcal{T}}_1,\widetilde{\mathcal{P}}\}$. Indeed, 

$$\widetilde{\Psi}_2 = \frac{1}{d} \Big(\widetilde{\mathcal{T}}_1 + (d-1)\widetilde{\mathcal{P}}\Big)\ . $$
Tomiyama map  $\widetilde{\mathcal{T}}_2 := \Lambda_{\frac{d}{d+1}} \in {\rm conv} \{\widetilde{\Psi}_1,\widetilde{\Psi}_2\}  $. One finds

$$ \widetilde{\mathcal{T}}_2 =  \frac{1}{d+1} \Big( (d-1) \widetilde{\Psi}_1 + 2 \widetilde{\Psi}_2 \Big) .  $$

\end{corollary}   



\begin{figure}[htbp]
\centering
\begin{tikzpicture}
\begin{axis}[
    axis lines=middle,
    xmin=-0.10, xmax=1.65,     
    ymin=-0.95, ymax=1.45,     
    xlabel={$\,\mu$},
    ylabel={$\,\nu$},
    xtick=\empty,
    ytick=\empty,              
    width=0.52\textwidth,
    height=0.46\textwidth,
    clip=false,
]

\pgfmathsetmacro{\d}{4}                
\pgfmathsetmacro{\mud}{\d/(\d-1)}      
\pgfmathsetmacro{\numin}{-2/(\d-1)}    
\pgfmathsetmacro{\mutwo}{\d/(\d+1)}    

\pgfplotsset{
  xtick={0,\mud},
  xticklabels={$0$},
}


\addplot[
  fill=blue!18,
  draw=blue!65!black,
  thick
] coordinates {
  (\mud,0)        
  (\mud,\numin)   
  (0,0)           
  (0,\mud)        
  (\mud,0)
};

\addplot[
  fill=green!30,
  draw=green!55!black,
  thick
] coordinates {
  (\mud,0)        
  (\mud,\numin)   
  (0,1)           
  (\mud,0)
};

\addplot[black, thick, dashed] coordinates {(0,0) (\mud,0)};

\addplot[only marks,mark=*,mark size=2pt] coordinates {
  (0,0)              
  (\mud,0)           
  (\mud,\numin)      
  (0,1)              
  (0,\mud)           
  (\mutwo,0)         
};

\node[below left]  at (axis cs:0,0) {$\widetilde{\mathcal{T}}_1$};
\node[above right] at (axis cs:\mud,0) {$\widetilde\Psi_0$};
\node[below right] at (axis cs:\mud,\numin) {$\widetilde\Psi_1$};
\node[above left]  at (axis cs:0,1) {$\widetilde\Psi_2$};
\node[left]        at (axis cs:0,\mud) {$\widetilde{\mathcal{P}}$};

\node[below] at (axis cs:\mutwo,0) {$\widetilde{\mathcal{T}}_2$};


\end{axis}
\end{tikzpicture}
\caption{Illustration of Corollary~3.3: quadrilateral of positive maps containing a green triangle of CP maps. The Tomiyama map $\widetilde{\mathcal{T}}_2$ corresponding to $(\frac{d}{d+1},0)$ is a convex combination of $\widetilde{\Psi}_1$ and $\widetilde{\Psi}_2$.   \label{FIG4}}
\end{figure}
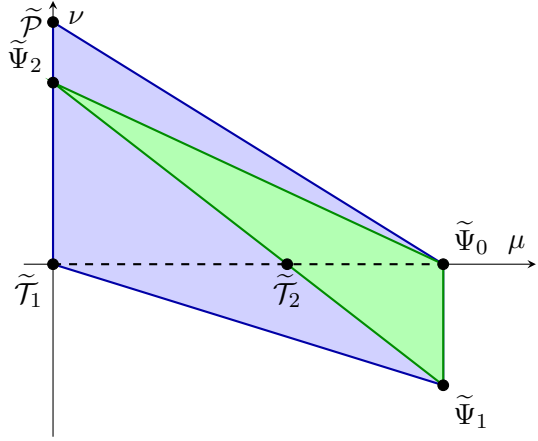


\section{Conclusions and outlook}   \label{SIV}

In this work, we analyzed two closely related two--parameter families of linear maps on matrix algebras, obtained as diagonal perturbations of the classical Tomiyama maps. Using the Choi matrix formalism combined with block--positivity techniques, we derived explicit necessary and sufficient conditions for positivity, complete positivity, and $k$--positivity in arbitrary dimension. In both families, the sets of completely positive maps admit a simple geometric description as triangles in the corresponding parameter spaces, while the sets of positive maps form larger convex regions containing these triangles.

For the family $\Phi_{\alpha,\beta}$, we showed that all positive maps are decomposable and identified the extremal points generating the positive region. This allowed us to formulate a conjecture describing the geometry of the $k$--positive region as a convex quadrilateral determined by the completely positive maps and the $k$--positive Tomiyama map. We proved this conjecture in several nontrivial cases, including $k=d-1$, and when $k$ divides the system dimension $d$, providing strong evidence for its validity in general.

A similar geometric picture emerges for the transposition--based family $\Lambda_{\mu,\nu}$. In this case, the completely positive region is generated by three explicitly identified maps, including the Holevo--Werner channel, while the positive region forms a quadrilateral containing the CP triangle. The appearance of Tomiyama maps as boundary points once again highlights their distinguished role in the structure of positive maps.

Several open problems remain. The full proof of the conjectured characterization of $k$--positive maps for arbitrary $k$ and $d$ is still missing.  It would also be interesting to investigate whether similar geometric descriptions persist for more general perturbations. Finally, the operational and entanglement--theoretic implications of the identified extremal maps, in particular their role as entanglement witnesses, deserve further study. Finally, it would be interesting to extend the analysis and study which maps $\Phi_{\alpha,\beta}$ and $\Lambda_{\mu,\nu}$ satisfy Schwarz inequality. Recently, it was shown how unital and trace-preserving $k$-positive maps may give rise to $k$-Kadison-Schwarz map \cite{F}. It would be interesting to generalize results of \cite{Tomiyama-II} for $k$-Kadison-Schwarz maps. 
Finally, it would be interesting to provide analysis of $k$-entanglement breaking maps generalizing recent results form \cite{kEB} .

\section*{Acknowledgements}

D.C. was supported by the Polish National Science
Centre project No. 2024/55/B/ST2/01781. A.B. acknowledges the support received for
this research from the research grant sanctioned
by the National Board for Higher Mathematics
(NBHM), Department of Atomic Energy (DAE),
Government of India, with sanction letter no:
02011/32/2025/NBHM(R.P)/R$\&$D II/9677; under seed money scheme from Birla Institute of Technology Mesra with sanction letter no: DRIE/SMS/DRIE-10917/2025-26/3857.

\appendix

\section{Appendix }


For $k=3$ and $d=k+1=4$ one obtains 
\[
A=\left[
\begin{array}{cccc|cccc|cccc}
 1 & -\tfrac13 & . & . & . & -\tfrac13 & -\tfrac13 & . & . & . & -\tfrac13 & -\tfrac13 \\
 -\tfrac13 & 1 & . & . & . & 1 & -\tfrac13 & . & . & . & -\tfrac13 & -\tfrac13 \\
 . & . & . & . & . & . & . & . & . & . & . & . \\
 . & . & . & . & . & . & . & . & . & . & . & . \\
 \hline
 . & . & . & . & . & . & . & . & . & . & . & . \\
 -\tfrac13 & 1 & . & . & . & 1 & -\tfrac13 & . & . & . & -\tfrac13 & -\tfrac13 \\
 -\tfrac13 & -\tfrac13 & . & . & . & -\tfrac13 & 1 & . & . & . & 1 & -\tfrac13 \\
 . & . & . & . & . & . & . & . & . & . & . & . \\
 \hline
 . & . & . & . & . & . & . & . & . & . & . & . \\
 . & . & . & . & . & . & . & . & . & . & . & . \\
 -\tfrac13 & -\tfrac13 & . & . & . & -\tfrac13 & 1 & . & . & . & 1 & -\tfrac13 \\
 -\tfrac13 & -\tfrac13 & . & . & . & -\tfrac13 & -\tfrac13 & . & . & . & -\tfrac13 & 1
\end{array}
\right]
\] and 
\[
B=\left[
\begin{array}{cccc|cccc|cccc}
 -\tfrac12 & \tfrac{2}{9} & . & . & \tfrac14 & \tfrac{2}{9} & \tfrac{2}{9} & . & . & . & \tfrac{2}{9} & \tfrac{2}{9} \\
 \tfrac{2}{9} & -\tfrac12 & . & . & . & -\tfrac34 & \tfrac{2}{9} & . & . & . & \tfrac{2}{9} & \tfrac{2}{9} \\
 . & . & \tfrac12 & . & . & . & \tfrac14 & . & . & . & . & . \\
 . & . & . & \tfrac12 & . & . & . & \tfrac14 & . & . & . & . \\
 \hline
 \tfrac14 & . & . & . & \tfrac12 & . & . & . & \tfrac14 & . & . & . \\
 \tfrac{2}{9} & -\tfrac34 & . & . & . & -\tfrac12 & \tfrac{2}{9} & . & . & \tfrac14 & \tfrac{2}{9} & \tfrac{2}{9} \\
 \tfrac{2}{9} & \tfrac{2}{9} & \tfrac14 & . & . & \tfrac{2}{9} & -\tfrac12 & . & . & . & -\tfrac34 & \tfrac{2}{9} \\
 . & . & . & \tfrac14 & . & . & . & \tfrac12 & . & . & . & \tfrac14 \\
 \hline
 . & . & . & . & \tfrac14 & . & . & . & \tfrac12 & . & . & . \\
 . & . & . & . & . & \tfrac14 & . & . & . & \tfrac12 & . & . \\
 \tfrac{2}{9} & \tfrac{2}{9} & . & . & . & \tfrac{2}{9} & -\tfrac34 & . & . & . & -\tfrac12 & \tfrac{2}{9} \\
 \tfrac{2}{9} & \tfrac{2}{9} & . & . & . & \tfrac{2}{9} & \tfrac{2}{9} & \tfrac14 & . & . & \tfrac{2}{9} & -\tfrac12
\end{array}
\right]
\]

\begin{equation*}
    \psi_1=(3,1,-1,1)^{\rm T} \ , \ \ \  \psi_2=(-2,2,2,-2)^{\rm T}
\end{equation*}
that is,

$$  \psi= (3,1,-1,1|-2,2,2,-2|1,-1,1,3)^{\rm T} \in \mathbb{C}^{12} . $$
For $k=4$ and $d=k+1=5$ one obtains 

\[
A=\left(
\begin{array}{ccccc|ccccc|ccccc|ccccc}
 1 & -\frac{1}{4} & . & . & . & . & -\frac{1}{4} & -\frac{1}{4} & . & . & . & . & -\frac{1}{4} & -\frac{1}{4} & . & . & . & . & -\frac{1}{4} & -\frac{1}{4} \\
 -\frac{1}{4} & 1 & . & . & . & . & 1 & -\frac{1}{4} & . & . & . & . & -\frac{1}{4} & -\frac{1}{4} & . & . & . & . & -\frac{1}{4} & -\frac{1}{4} \\
 . & . & . & . & . & . & . & . & . & . & . & . & . & . & . & . & . & . & . & . \\
 . & . & . & . & . & . & . & . & . & . & . & . & . & . & . & . & . & . & . & . \\
 . & . & . & . & . & . & . & . & . & . & . & . & . & . & . & . & . & . & . & . \\ \hline
 . & . & . & . & . & . & . & . & . & . & . & . & . & . & . & . & . & . & . & . \\
 -\frac{1}{4} & 1 & . & . & . & . & 1 & -\frac{1}{4} & . & . & . & . & -\frac{1}{4} & -\frac{1}{4} & . & . & . & . & -\frac{1}{4} & -\frac{1}{4} \\
 -\frac{1}{4} & -\frac{1}{4} & . & . & . & . & -\frac{1}{4} & 1 & . & . & . & . & 1 & -\frac{1}{4} & . & . & . & . & -\frac{1}{4} & -\frac{1}{4} \\
 . & . & . & . & . & . & . & . & . & . & . & . & . & . & . & . & . & . & . & . \\
 . & . & . & . & . & . & . & . & . & . & . & . & . & . & . & . & . & . & . & . \\ \hline
 . & . & . & . & . & . & . & . & . & . & . & . & . & . & . & . & . & . & . & . \\
 . & . & . & . & . & . & . & . & . & . & . & . & . & . & . & . & . & . & . & . \\
 -\frac{1}{4} & -\frac{1}{4} & . & . & . & . & -\frac{1}{4} & 1 & . & . & . & . & 1 & -\frac{1}{4} & . & . & . & . & -\frac{1}{4} & -\frac{1}{4} \\
 -\frac{1}{4} & -\frac{1}{4} & . & . & . & . & -\frac{1}{4} & -\frac{1}{4} & . & . & . & . & -\frac{1}{4} & 1 & . & . & . & . & 1 & -\frac{1}{4} \\
 . & . & . & . & . & . & . & . & . & . & . & . & . & . & . & . & . & . & . & . \\ \hline
 . & . & . & . & . & . & . & . & . & . & . & . & . & . & . & . & . & . & . & . \\
 . & . & . & . & . & . & . & . & . & . & . & . & . & . & . & . & . & . & . & . \\
 . & . & . & . & . & . & . & . & . & . & . & . & . & . & . & . & . & . & . & . \\
 -\frac{1}{4} & -\frac{1}{4} & . & . & . & . & -\frac{1}{4} & -\frac{1}{4} & . & . & . & . & -\frac{1}{4} & 1 & . & . & . & . & 1 & -\frac{1}{4} \\ 
 -\frac{1}{4} & -\frac{1}{4} & . & . & . & . & -\frac{1}{4} & -\frac{1}{4} & . & . & . & . & -\frac{1}{4} & -\frac{1}{4} & . & . & . & . & -\frac{1}{4} & 1 \\
\end{array}
\right)\]

\[B=\left(
\begin{array}{ccccc|ccccc|ccccc|ccccc}
 -\frac{3}{5} & \frac{3}{16} & . & . & . & \frac{1}{5} & \frac{3}{16} & \frac{3}{16} & . & . & . & . & \frac{3}{16} & \frac{3}{16} & . & . & . & . & \frac{3}{16} & \frac{3}{16} \\
 \frac{3}{16} & -\frac{3}{5} & . & . & . & . & -\frac{4}{5} & \frac{3}{16} & . & . & . & . & \frac{3}{16} & \frac{3}{16} & . & . & . & . & \frac{3}{16} & \frac{3}{16} \\
 . & . & \frac{2}{5} & . & . & . & . & \frac{1}{5} & . & . & . & . & . & . & . & . & . & . & . & . \\
 . & . & . & \frac{2}{5} & . & . & . & . & \frac{1}{5} & . & . & . & . & . & . & . & . & . & . & . \\
 . & . & . & . & \frac{2}{5} & . & . & . & . & \frac{1}{5} & . & . & . & . & . & . & . & . & . & . \\ \hline
 \frac{1}{5} & . & . & . & . & \frac{2}{5} & . & . & . & . & \frac{1}{5} & . & . & . & . & . & . & . & . & . \\
 \frac{3}{16} & -\frac{4}{5} & . & . & . & . & -\frac{3}{5} & \frac{3}{16} & . & . & . & \frac{1}{5} & \frac{3}{16} & \frac{3}{16} & . & . & . & . & \frac{3}{16} & \frac{3}{16} \\
 \frac{3}{16} & \frac{3}{16} & \frac{1}{5} & . & . & . & \frac{3}{16} & -\frac{3}{5} & . & . & . & . & -\frac{4}{5} & \frac{3}{16} & . & . & . & . & \frac{3}{16} & \frac{3}{16} \\
 . & . & . & \frac{1}{5} & . & . & . & . & \frac{2}{5} & . & . & . & . & \frac{1}{5} & . & . & . & . & . & . \\
 . & . & . & . & \frac{1}{5} & . & . & . & . & \frac{2}{5} & . & . & . & . & \frac{1}{5} & . & . & . & . & . \\ \hline
 . & . & . & . & . & \frac{1}{5} & . & . & . & . & \frac{2}{5} & . & . & . & . & \frac{1}{5} & . & . & . & . \\
 . & . & . & . & . & . & \frac{1}{5} & . & . & . & . & \frac{2}{5} & . & . & . & . & \frac{1}{5} & . & . & . \\
 \frac{3}{16} & \frac{3}{16} & . & . & . & . & \frac{3}{16} & -\frac{4}{5} & . & . & . & . & -\frac{3}{5} & \frac{3}{16} & . & . & . & \frac{1}{5} & \frac{3}{16} & \frac{3}{16} \\
 \frac{3}{16} & \frac{3}{16} & . & . & . & . & \frac{3}{16} & \frac{3}{16} & \frac{1}{5} & . & . & . & \frac{3}{16} & -\frac{3}{5} & . & . & . & . & -\frac{4}{5} & \frac{3}{16} \\
 . & . & . & . & . & . & . & . & . & \frac{1}{5} & . & . & . & . & \frac{2}{5} & . & . & . & . & \frac{1}{5} \\ \hline
 . & . & . & . & . & . & . & . & . & . & \frac{1}{5} & . & . & . & . & \frac{2}{5} & . & . & . & . \\
 . & . & . & . & . & . & . & . & . & . & . & \frac{1}{5} & . & . & . & . & \frac{2}{5} & . & . & . \\
 . & . & . & . & . & . & . & . & . & . & . & . & \frac{1}{5} & . & . & . & . & \frac{2}{5} & . & . \\
 \frac{3}{16} & \frac{3}{16} & . & . & . & . & \frac{3}{16} & \frac{3}{16} & . & . & . & . & \frac{3}{16} & -\frac{4}{5} & . & . & . & . & -\frac{3}{5} & \frac{3}{16} \\
 \frac{3}{16} & \frac{3}{16} & . & . & . & . & \frac{3}{16} & \frac{3}{16} & . & . & . & . & \frac{3}{16} & \frac{3}{16} & \frac{1}{5} & . & . & . & \frac{3}{16} & -\frac{3}{5} \\
\end{array}
\right)\]

\begin{equation*}
    \psi_1=(4,1,-1,1,-1)^{\rm T} \ , \ \ \  \psi_2=(-3,3,2,-2,2)^{\rm T}  \ , 
\end{equation*}
that is,

$$  \psi= (4,1,-1,1,-1|-3,3,2,-2,2|2,-2,2,3,-3|-1,1,-1,1,4)^{\rm T} \in \mathbb{C}^{20} . $$


\end{document}